\documentclass[12pt, a4paper]{amsart}

\usepackage[backend=bibtex,url=false,giveninits=true,isbn=false,doi=true,maxbibnames=99,sortcites]{biblatex}	
\usepackage{amssymb}	
\usepackage{enumerate}
\usepackage{mathtools}
\usepackage{xcolor}

\usepackage[colorlinks=true,linkcolor=red!50!black,citecolor=green!50!black,urlcolor=blue!50!black]{hyperref}
\usepackage[capitalise]{cleveref}

\definecolor{cyan}{cmyk}{1,0,0,0}

\newtheorem{Thm}{Theorem}[section]

\newtheorem{Cor}[Thm]{Corollary}

\newtheorem{Prop}[Thm]{Proposition}

\newtheorem{Lem}[Thm]{Lemma}

\newtheorem{customthm}{Theorem}[section]

\newcounter{claim}[Thm]

\theoremstyle{definition}

\theoremstyle{remark}

\newtheorem{remark}[Thm]{Remark}

\crefname{customthm}{Theorem}{Theorems}
\crefname{Cor}{Corollary}{Corollaries}
\crefname{Conj}{Conjecture}{Conjectures}
\crefname{Prop}{Proposition}{Propositions}
\crefname{Qu}{Question}{Questions}
\crefname{Lem}{Lemma}{Lemmas}
\crefname{Def}{Definition}{Definitions}
\crefname{Not}{Notation}{Notation}
\crefname{Ex}{Example}{Examples}
\crefname{Rem}{Remark}{Remarks}
\crefname{remark}{Remark}{Remarks}		

\numberwithin{equation}{section}

\newcommand{\Aut}{\operatorname{Aut}}

\newcommand{\Syl}{\operatorname{Syl}}

\newcommand{\Inn}{\operatorname{Inn}}

\newcommand{\Sym}{\operatorname{Sym}}

\newcommand{\Alt}{\operatorname{Alt}}

\newcommand{\PSU}{\operatorname{PSU}}

\newcommand{\PSL}{\operatorname{PSL}}
\newcommand{\PGL}{\operatorname{PGL}}
\newcommand{\GL}{\operatorname{GL}}
\newcommand{\SL}{\operatorname{SL}}

\newcommand{\SU}{\operatorname{SU}}

\newcommand{\PSp}{\operatorname{PSp}}

\renewcommand{\epsilon}{\varepsilon}

\newcommand{\ov}{\overline}

\def \ov {\overline}

\DeclareMathOperator{\E}{E}
\DeclareMathOperator{\D}{D}
\DeclareMathOperator{\F}{F}
\DeclareMathOperator{\G}{G}

\DeclareMathOperator{\GF}{GF}

\DeclareMathOperator{\J}{J}
\DeclareMathOperator{\PGU}{PGU}

\DeclareMathOperator{\Sz}{{}^2 B_2}
\DeclareMathOperator{\Ree}{{}^2 G_2}
\DeclareMathOperator{\BigRee}{{}^2 F_4}
\DeclareMathOperator{\triality}{{}^3 D_4}	

\DeclarePairedDelimiter{\gen}{\langle}{\rangle}
\DeclarePairedDelimiter{\abs}{\lvert}{\rvert}

\renewbibmacro{in:}{}
\DeclareFieldFormat[article]{title}{#1}

\begin{document}

\title {Squares of conjugacy classes and a variant on the Baer--Suzuki Theorem}
 \author{Chris Parker}
\author{Jack Saunders}
\address{Chris Parker\\
School of Mathematics\\
University of Birmingham\\
Edgbaston\\
Birmingham B15 2TT\\
United Kingdom} \email{c.w.parker@bham.ac.uk}
\address{Jack Saunders\\
The University of Western Australia\\ 35 Stirling Highway\\ Perth\\
WA 6009\\
Australia}
\email{jack-saunders@hotmail.co.uk}

\date{\today}

\begin{abstract}
  For $p$ a prime, $G$ a finite group and $A$ a normal subset of elements of order $p$, we prove that if $A^2 = \{ab \mid a, b \in A\}$ consists of $p$-elements then $Q = \gen{A}$ is soluble. Further, if $O_p(G) = 1$, we show that $p$ is odd, $F(Q)$ is a non-trivial $p'$-group and $Q/F(Q)$ is an elementary abelian $p$-group. We also provide examples which show this conclusion is best possible.
\end{abstract}

\maketitle \pagestyle{myheadings}

\markright{{\sc }} \markleft{{\sc Chris Parker and Jack Saunders }}

{\let\thefootnote\relax\footnote{{
The first author thanks the Isaac Newton Institute for Mathematical Sciences for support during the programme \emph{Groups, representations and applications: new perspectives} when work on this paper was undertaken. This work was supported by \textbf{EPSRC Grant Number EP/R01 146 04/1}. The second author thanks the department at UWA for their hospitality and acknowledges the support of the \textbf{Australian Research Council Discovery Project Grant DP190101024.}}}

\section{Introduction}

For $A$ and $B$ subsets of a group $G$, the product of $A$ and $B$ is defined to be $AB=\{ab\mid a\in A,b\in B\}$. Obviously, if $A$ and $B$ are unions of conjugacy classes, then so is $AB$. Products of conjugacy classes in finite and algebraic groups have been studied by many authors \cite{ClassPowers,Gow,GM,GMT,GuralnickNavarro,GuralnickRobinson, Xiao}. Particularly  for finite simple groups $G$, questions about products of conjugacy classes have led to significant conjectures. Prominent amongst these are two extreme cases: Thompson's conjecture which asserts that there exists a conjugacy class $C$ of $G$ such that $G=C^2$ (so that $C^2$ attains its maximal possible size) and the Arad--Herzog conjecture which speculates that if $A$ and $B$ are conjugacy classes of $G$, then $AB$ is not a conjugacy class (so it cannot be as small as possible). In this vein of research, Guralnick and Navarro \cite{GuralnickNavarro} have demonstrated that if \(A = a^G\) is such that \(A^2\) is a conjugacy class, then \(\gen{A} = [a, G]\) is soluble while in \cite{ClassPowers} this is generalised to the statement that if \(A^n \subseteq \{1\} \cup x^G\) for some $x\in G$ and $n \ge 2$, then again \(\gen{A}\) is soluble. Our main results contribute to this type of research by investigating groups $G$ which have conjugacy classes $A$ of elements of order $p$, $p$ a prime, which have $A^2$ consisting of $p$-elements. Our main result, which uses the classification of finite simple groups, is as follows.

\begin{customthm}\label{thm:A}	
  Suppose $p$ is a prime, $G$ is a finite group, $A$ is a non-empty normal subset of elements of order $p$ in $G$ and every member of $A^2$ is a \(p\)-element. Then, setting $Q=\gen{A}$, $Q$ is soluble, and, if $O_p(G)=1$,
  then $p$ is odd, $F(Q)$ is a non-trivial $p'$-group and $Q/F(Q)$ is an elementary abelian $p$-group.
\end{customthm}

 \cref{thm:A} generalizes the Baer--Suzuki Theorem which asserts that if any two elements of a conjugacy class $A$ of a finite group $G$ generate a nilpotent subgroup, then $\gen{A}$ is a normal nilpotent subgroup of $G$. If $A$ is a conjugacy class of involutions, then any two elements of $A$ generate a nilpotent subgroup if and only if they generate a dihedral $2$-group, which is if and only if their product is a $2$-element. Thus, $A$ generates a normal $2$-subgroup of $G$ if and only if $A^2=\{ab\mid a,b\in A\}$ consists of $2$-elements. This means that the content of \cref{thm:A} concerns elements of odd order.

The alternating group $\Alt(4)$ is generated by a conjugacy class $A$ of elements of order $3$. In this case $A^2$ is also a conjugacy class of elements of order $3$. This shows that we cannot replace soluble by nilpotent in the conclusion of \cref{thm:A}.

In \cite{GM} Guralnick and Malle prove that if $A$ is a conjugacy class of $p$-elements, and $AA^{-1}$ consists just of $p$-elements, then $A \subseteq O_p(G)$.

\cref{thm:A} is similar to a result of Xiao \cite[Theorem 2]{Xiao} also proved later by Guralnick and Robinson \cite[Theorem A]{GuralnickRobinson} where it is demonstrated that if $A$ is a conjugacy class of elements of order $p$ such that $x^{-1}y$ is a $p$-element for every $x,y \in A$, then $\gen{A}$ is nilpotent. As a further generalization of the Baer--Suzuki Theorem, we mention Guest's Theorem \cite{Guest} which asserts that if $p\ge 5$ is a prime and $A$ is a conjugacy class of elements of order $p$, then $\gen{A}$ is soluble if and only if $\gen{x,y}$ is soluble for all $x,y\in A$.

 We point out the following corollaries:

 \begin{Cor}\label{cor:cor11}
  Suppose that $p$, $A$ and $G$ are as in \cref{thm:A} and assume that $O_p(G)=1$. Then $p$ is odd and every element of $A^2$ has order $p$.
  \end{Cor}

 Evidently, if $O_p(G)\ne 1$, we cannot control the order of elements in $A^2$. For groups generated by a conjugacy class we have the following result.

\begin{Cor}\label{cor:cor3}
   Suppose that $p$, $A$ and $G$ are as in \cref{thm:A} and assume that $O_p(G)=1$. If $A$ is a conjugacy class and $G= \gen{A}$, then $p$ is odd, $G$ is a Frobenius group with Frobenius complements $\gen{a}$ for $a \in A$. Furthermore, every element of $A^2$ has order $p$.
\end{Cor}}

 Notice that $\Alt(4)\times \Alt(4) \times \Alt(4)$ has a diagonal subgroup of index $3$ which is not a direct product of $\Alt(4)$ subgroups and is generated by a union of two conjugacy classes $A$ of elements of order $3$ with every element in $A^2$ of order $3$. This shows that when $O_p(G) \ne 1$, we cannot conclude in \cref{thm:A} that $\langle A\rangle$ is a direct product of Frobenius groups as might be hoped when considering \cref{cor:cor3}.

The paper develops as follows. In \cref{sec:preliminary}, we present the group theoretical results which support our proofs and then \cref{sec:sol} investigates soluble groups satisfying the hypothesis of \cref{thm:A}. After this section, the main goal is to demonstrate that $\langle A\rangle$ is soluble and this part of the proof requires the classification of finite simple groups. \cref{sec:reduction,sec:Lie type,sec:computer} are devoted to the machinery for the proof of what could be considered the main theorem of the article.

\begin{Thm}\label{thm:main}
  Suppose $p$ is a prime, $G$ is a finite group and $A$ is a normal subset of elements of order $p$ in $G$. If every member of $A^2$ is a \(p\)-element, then $\gen{A}$ is a soluble normal subgroup of $G$.
\end{Thm}

For almost simple groups $G$, Guralnick, Malle and Tiep \cite{GMT} show that, if $c$ and $d$ are elements of prime order $p\ge 5$ in $G$ then there exists $g \in G$ such that $cd^g$ is not a $p$-element. We could of course apply this result to abbreviate some of the arguments used in the proof of \cref{thm:main} when $p \ge 5$; however this would still leave the $p=3$ case to handle and presenting a complete analysis does not significantly lengthen the arguments.

That \cref{thm:main} cannot be generalized to the situation where every element of $A^2$ is either a $p$-element or an $r$-element with $p$ and $r$ distinct primes and with both possibilities arising is demonstrated by the symmetric groups with $A$ the conjugacy class of transpositions.

We remark that small groups are easily handled using {\sc Magma} and {\sc GAP}\cite{Magma, GAP4} and some {\sc GAP} code is included in \cref{sec:appendix}. The instances where we apply computer calculations are presented in \cref{sec:computer} and sometimes we reference forward to \cref{prop:sporadics}.

Finally, in \cref{sec:proofs} we provide the proofs of the results stated in the introduction.
\medskip
\renewcommand{\abstractname}{Acknowledgements}
\begin{abstract}
We thank the referee for the key observation in the proof of \cref{lem:dih calc}. This  provided a stronger result with a much nicer proof than our original  statement about wreath products. \cref{prop:invert stuff} was also provided by the referee.
\end{abstract}

\section{Preliminary results}\label{sec:preliminary}

This section is devoted to the presentation of general group theoretic results which will be used to prove the main theorem. The first result allows us to demonstrate that wreath products of non-abelian simple groups cannot provide a counter example to \cref{thm:main}.

\begin{Lem} \label{lem:dih calc}
  Let \(p\) be a prime and \(H\) be a non-abelian group of order not divisible by \(p\). Suppose that \(W \cong H \wr C\) where $C$ is cyclic of order $p$ and let \(x \in W\) have order \(p\). Then there exists \(w \in W\) such that \(x^w x\) is not a \(p\)-element.
\end{Lem}

\begin{proof}
  If \(p = 2\) then \(x^w x\) lies in \(O_{p'}(W)\) and can be chosen to be non-trivial. Thus assume that \(p > 2\).

  Let \(N = O_{p'}(W) \cong H_1 \times \dots \times H_p\) where \(H_i \cong H\) for all \(i\) and assume that \((h_1, \ldots, h_p)^x = (h_p, h_1, \ldots, h_{p-1})\). Note that \((h_1, \ldots, h_p)x\) is a \(p\)-element if and only if \(h_1 \dots h_p = 1\) and similarly \((h_1, \ldots, h_p)x^2\) is a \(p\)-element if and only if \(h_1 h_3 \dots h_p h_2 \dots h_{p-1} = 1\). 
  Since \(H\) is non-abelian, we may choose \(a\), \(b \in H\) such that \([a,b]\neq 1\). Then for \(w = (a, b^{-1}, 1, \ldots, 1)\) we have that
  \[x^w x = w^{-1} w^{x^{-1}} x^2 = (a^{-1} b^{-1}, b, 1, \ldots, 1, a) x^2\]
  is a \(p\)-element if and only if \(a^{-1}b^{-1} a b= [a, b] = 1\). This contradicts the choice of $a$ and $b$  and thus \(x^w x\) is not a \(p\)-element.
\end{proof}

\begin{Lem}  \label{prop:invert stuff}
  Suppose \(p\) is an odd prime and \(G\) is a non-abelian finite simple group. Then \(G\) contains a non-abelian \(p'\)-subgroup.
\end{Lem}

\begin{proof}
  If the Sylow 2-subgroups of \(G\) are non-abelian then the result is clear. The non-abelian simple groups with abelian Sylow $2$-subgroups are \( \PSL_2(q)\) with $q=2^n\ge 4$ or $q\equiv 3,5 \pmod 8$ with $q\ne 3$, \(\Ree(3^{2n+1})\)  with $n\ge 1$  and  \(\J_1\). By induction we may assume that every proper subgroup of \(G\) is non-soluble, as $\PSL_2(8) \cong \Ree(3)' < \Ree(3^{2n+1})$ and $\PSL_2(11)$ is isomorphic to a subgroup of $\J_1$, we may assume that \(G=\PSL_2(q)\). However, in this case there are non-abelian dihedral subgroups with 2 being the greatest common divisor of their orders and so we are done.
\end{proof}

The next two lemmas are used in the final arguments which establish \cref{thm:A} in \cref{sec:proofs}.

\begin{Lem}\label{lem:is abundant}
  Suppose that $P$ is a $p$-group and $a\in P$ has order $p$. Assume there exists an elementary abelian subgroup of $P$ which is normalized but not centralized by $a$. Then there is an elementary abelian subgroup $X$ of $P$ of order $p^2$ such that $\abs{a^P \cap X}=p$ and $X \cap P' \ne 1$.
\end{Lem}

\begin{proof}
  Assume that $P$ is a minimal counterexample to the claim. Choose $E \le P$ elementary abelian of minimal order such that $a \not \in C_P(E)$ and $E^a=E$. Then, the minimality of $P$ yields $P=E\gen{a}$. Since $[E,\gen{a}]< E$, $[E,\gen{a}] \le C_E(a)<E$. Let $e\in E\setminus C_E(a)$, then $\gen{[e,a],e}$ is normalized by not centralized by $a$. Hence $E=\gen{[e,a],e}$ has order $p^2$ by the minimality of $E$ and $P$ is extraspecial of order $p^3$ (dihedral for \(p = 2\)). Thus $X=\gen{[e,a],a}$ is normalized but not centralized by $P$ and $C_{P}(X)= X$. It follows that $a^P\subseteq X$ and $\abs{a^P\cap X}=p$. Finally, we note that $[e,a]\in X\cap P'$ and we have a contradiction.
\end{proof}

\begin{Lem}\label{lem:generation abelian}
  Suppose that $p$ is a prime, $G$ is a $p$-group and $A$ is a non-empty normal subset of $G$ consisting of elements of order $p$. Assume that for all $a\in A$, $a$ centralizes every elementary abelian $p$-subgroup of $G$ which it normalizes. Then $\gen {A}$ is an elementary abelian normal subgroup of $G$.
\end{Lem}

\begin{proof}
  Suppose that $G$ is a counterexample of minimal order and let \(a \in A\). If $G=\gen{a}$ then we have an immediate contradiction. Let $M$ be a maximal subgroup of $G$ which contains $a$. Then $A \cap M$ is non-empty and the minimal choice of $G$ gives that $\gen{A \cap M}$ is elementary abelian. Since $M$ is normal in $G$, $A\cap M$ is a non-empty normal subset of $G$. Hence $\gen{A \cap M}$ is normalized, and thus centralized by every element of $A$. In particular for any \(b \in A\), $b$ centralizes $a\in \gen{A \cap M}$ and so $a \in Z(\gen{A})$. Since this is true for all $a \in A$, $\gen{A}$ is elementary abelian. This contradiction proves the lemma.
\end{proof}

We continue this section with an elementary consequence of a theorem of Gross, which applies to all finite simple groups.

\begin{Lem}\label{lem:GrossApp}
	Suppose that $r$ is an odd prime, $K=\Inn(K)$ is a finite simple group of order divisible by $r$ and $\alpha \in \Aut(K)\setminus K$ has order $r$. Then $C_{K\gen{\alpha}}(\alpha) \cap \alpha^K \not \subseteq \gen{\alpha}$.
\end{Lem}

\begin{proof}
	Let $R \in \Syl_r(K\gen{\alpha})$ with $\alpha \in R$ and set $P= R \cap K$. As $r$ is odd, \cite[Theorem B]{Gross} implies that $\alpha \not \in C_R(P)=Z(P)$. In particular, $C_{P}( \alpha ) < P$ and there exists $x \in N_{P}(C_{P}(\alpha ))\setminus C_{P}(\alpha)$ such that $\gen{\beta} = \gen{\alpha}^x \leq Z(C_P(\alpha))$. If $ \gen{\alpha} = \gen{\beta}$, then $[\alpha,x] =\alpha^{-1}\alpha^x \in K \cap \gen{\alpha}$. Since $\alpha$ has order $r$, $K \cap \gen{\alpha} = 1$, and thus $x \in C_P(\alpha)$, a contradiction.
	Hence $ \gen{\alpha} \ne \gen{\beta}$ and this proves the claim.
\end{proof}

\begin{Lem}\label{lem:solvpara}
	Suppose that $K$ is the derived group of a Lie type group $L$ of rank at least $2$. Assume that $P \le L$ is a maximal parabolic subgroup of $L$. If $P$ is soluble, then $K$ is one of the following groups: $\PSL_3(2)$, $\PSL_3(3)$, $\PSL_4(2)$, $\PSL_4(3)$, $\PSU_4(2)$, $\PSU_4(3)$, $ \PSU_5(2)$, $ \PSp_4(2)'$, $ \PSp_4(3)$, $\PSp_6(2)$, $\PSp_6(3)$, $\Omega_7(3)$, $\Omega_8^+(2)$, $\Omega_8^+(3)$, $\G_2(2)'$, $\G_2(3)$, $\triality(2)$, $ \triality(3)$, $\BigRee(2)'$.
\end{Lem}

\begin{proof}
	This can be extracted from \cite[Lemma 5.6]{Burness}. Note that \cite[Lemma 5.6]{Burness} allows graph automorphisms and so includes the additional groups $\PSL_5(2)$, $\PSL_5(3)$, $\PSL_6(2)$, $\PSL_6(3)$ and $\F_4(2)$ that are not itemised in our lemma. In addition, we remark that $\PSU_5(3)$ does not have any soluble maximal parabolic subgroups.
\end{proof}

Our last result of this section exploits the work of Guralnick and Navarro mentioned in the introduction and deals with a special case of our main theorems.

\begin{Thm}\label{thm:GN}
  Suppose that $G$ is a finite group, $N$ is a normal subgroup of $G$ of index $p$ and $\gen{a} \in \Syl_p(G)$ has order $p$. Set $A= a^G$ and assume that every element of $A^2$ is a $p$-element. Then $\gen{A}$ is a Frobenius group. In particular, $[N,a]$ is soluble.
\end{Thm}

\begin{proof}
  Since the Sylow $p$-subgroups of $G$ have order $p$ and $\abs{G:N}=p$, $N$ is a $p'$-group and $G= \gen{a} N> N$. Let $b\in A$. Then
	$ab$ is a $p$-element and $aNbN= a^2N$.
	In particular, as $p$ does not divide $\abs{N}$, Sylow's Theorem gives $g \in G$ such that $ab \in \gen{a^g}$. Since
	$$abN= a^2N= (a^g)^2N,$$
  we have $(a^g)^{-2} ab\in N\cap \gen{a^g}=1$ and so $ab = (a^g)^{2}=(a^2)^g$. In particular, $A^2 \subseteq (a^2)^G$ and we get that $A^2$ is a conjugacy class. Now \cite[Theorem A]{GuralnickNavarro} implies $[G,a]= [N,a]$ is soluble and $A= a[N,a]$. Set $X= \gen{A}= \gen{a}[N,a]$. Then as $A= a[N,a]$, $C_X(a)= \gen{a}$.
  Hence acting on the cosets of $\gen{a}$, $X$ is a Frobenius group with complement $\gen{a}$. This proves the theorem.
\end{proof}

\section{The structure of \texorpdfstring{$\gen{A}$}{<A>} when $G$ is soluble}\label{sec:sol}

We now fix a prime $p$, finite group $G$ and a non-empty normal subset $A$ of elements of order $p$ in $G$ such that every element of $A^2$ is a \(p\)-element. We also assume that $O_p(G)=1$. As a consequence, $p$ is odd, for otherwise we know that $\langle A\rangle$ is a $2$-group by the Baer--Suzuki Theorem and thus $A \subset O_2(G)$ which is impossible as $A$ is non-empty.

\begin{Lem}\label{lem:not abundant}
  Suppose that $G$ is soluble, $O_p(G)=1$ and $a \in A\setminus G'$. If $X=\gen{c, a}$ is elementary abelian of order $p^2$ and $c\in G'$, then $\abs{a^G \cap X}<p$.
\end{Lem}

\begin{proof}
  Set $Q=O_{p'}(G)$. Then $F(G)\le Q$ and $C_G(Q) \le Q$. We may as well assume that $A= a^G$. Suppose that $\abs{A \cap X} \ge p$. Since every conjugate of $a$ lies in $aG'$ and $c \in G'$, $$A\cap X= \{ac^j\mid 0\le j \le p-1\}$$ and $$X=\gen{c} \cup\bigcup_{x\in A\cap X} \gen{x}.$$
  Since $\abs{C_{Q}(a)}= \abs{C_{Q^g}(a^g)}= \abs{C_Q(a^g)}$ for all $g\in G$, $$\abs{C_Q(a)}=\abs{C_{Q}(x)}$$ for all $x \in A\cap X$.
  Because $Q= \gen{C_Q(y) \mid y \in X^\#}$ by \cite[Theorem 5.3.16]{Gorenstein} and the cyclic groups generated by elements of $X\cap A$ are all different, we observe that, if $C_Q(a)=1$, we have $Q= C_Q(c)$ which is impossible as $C_G(Q) \le Q$ and $Q$ is a $p'$-group. Hence $C_Q(x)\ne 1$ for all $x \in A\cap X$.
  Notice that $$\gen{a^2c} = \gen{(a^2c)^{(p+1)/2}} = \gen{ac^{(p+1)/2}}$$ and so $C_{Q}( a^2c)\ne 1$.
  Pick $t \in C_{Q}( a^2c)^\#$. As $X$ centralizes $a^2c$ and normalizes $Q$, $X$ normalizes $C_{Q}( a^2c)$. Therefore
  $$[ac,t] \in [ac,C_{Q}( a^2c)]\le C_{Q}( a^2c)$$
  and thus
  \begin{eqnarray*}
    (a (ac)^t)^p  &=& (a^2cc^{-1}a^{-1}(ac)^t)^p = (a^2c[ac,t])^p\\
                  &=& (a^2c)^p[ac,t]^p = [ac,t]^p.
  \end{eqnarray*}
  Because $a(ac)^t\in A^2$ and $Q$ is a $p'$-group, we deduce that $[ac,t]=1$. Consequently, $C_Q(a^2c)\le C_Q(ac)$. It follows that $C_Q(a^2c)$ is centralized by $\gen{ac,a^2c}= X$. As $\abs{C_{Q}(a)}= \abs{C_{Q}(x)} $ for all $x \in A\cap X$, we obtain, for all $x\in A \cap X$, $$C_Q(a)= C_Q(x)=C_Q(X) \le C_Q(c).$$ Since $Q= \gen{C_Q(y) \mid y\in X^\#}$ and $X=\gen{c} \cup\bigcup_{x\in A\cap X} \gen{x}$, again we conclude that $c \in C_G(Q)\le Q$, contrary to $Q$ being a $p'$-group. This contradiction proves the lemma.
\end{proof}

\begin{Lem}\label{lem:class generated}
  Suppose that $A$ is a conjugacy class of $G$, $G=\gen{A}$ is soluble and $O_p(G)=1$. Then $G $ is a Frobenius group, and, for $a \in A$, $\gen{a}$ is a Frobenius complement.
\end{Lem}

\begin{proof}
  Fix $a \in A$ and set $F= O_{p'}(G)$. Then $F < G$ and, as $O_p(G)=1$, $F \ne 1$. By \cref{thm:GN}, we may assume that $\abs{G:F} > p$.

  Assume that $G$ is a counterexample of minimal order and use the bar notation $\ov G= G/F$. Denote by $\ov Q= O_p(\ov G)$ (so the full preimage of $\ov Q$ is denoted by $Q$). Then the minimality of $G$ and the fact that $\abs{\ov G} \ne p$ implies $\ov Q\ne 1$ and $ G / Q $ is a Frobenius group with complement $\gen{ {a}}Q/Q$.
  Let $L \ge Q$ be such that $L/Q= O_{p'}(G/Q)$. As $G/Q$ is a Frobenius group with complement $\gen{a}Q/Q$, $L/Q= F(G/Q)$ has index $p$ in $G$ and is the Frobenius kernel of $G/Q$. Fix a Sylow $p$-subgroup $\ov P$ of $\ov G$ containing $\ov a$. Then $\ov Q $ has index $p$ in $\ov P$ and $\ov P = \gen{\ov a}\ov Q$.

 \cref{lem:is abundant,lem:not abundant} together imply that $\ov a$ centralizes every elementary abelian subgroup of $\ov{P}$ which it normalizes.
  By \cref{lem:generation abelian}, setting $\ov X=\gen{\ov{a}^{\ov{G}}\cap \ov{P} }$, we have that $\ov X$ is an elementary abelian normal subgroup of $\ov {P}$. Hence $$[\ov{Q},\ov a, \ov a] \le [\ov Q,\ov X,\ov X] \le [\ov X,\ov X]=1.$$
  In particular, $\ov G/C_{\ov G}(\ov Q)$ is not $p$-stable \cite[Definition 25.3]{GLS2}.
  Therefore \cite[Proposition 25.6]{GLS2} yields $\SL_2(3)$ is involved in $\ov G/C_{\ov{G}}(\ov Q)$ and hence also in the Frobenius group $G/Q$. In particular, $\ov a$ has order $p=3$. Since $\abs{Z(\SL_2(3))}=2$ and $G/Q$ is a Frobenius group, on the one hand $\gen{a}Q/Q$ acts fixed-point-freely on $L/Q$ whereas, on the other, it must have centralizer of order at least $2$. This contradiction proves the lemma.
\end{proof}

\section{The reduction of \cref{thm:main} to simple groups}	\label{sec:reduction}

For the proof of \cref{thm:main}, assume that $G$ is a counterexample with $\abs{G}+\abs{A}$ minimal.

\begin{Lem}\label{lem:basic1}
	If $H$ is a proper subgroup of $G$ and $A \cap H \not=\emptyset$, then $\gen{A \cap H}$ is a normal soluble subgroup of $H$. In particular, $G=\gen{A}$.
\end{Lem}

\begin{proof}
	We have $\abs{H}+\abs{A\cap H} < \abs{G}+ \abs{A}$ and so $\gen{A \cap H}$ is a normal soluble subgroup of $H$ by our inductive hypothesis. If $\gen{A} \ne G$, we may take $H=\gen{A}$ and deduce that $\gen{A}$ is a soluble subgroup, which, as $A$ is a normal subset of $G$, is normal in $G$.
\end{proof}

\begin{Lem}	\label{lem:lem3}
	If $N$ is a non-trivial proper normal subgroup of $G$, then $G/N$ is soluble. In particular, $N$ is not soluble.
\end{Lem}

\begin{proof}
	We have $A/N= \{aN\mid a\in A \}$ is a union of conjugacy classes in $G/N$. Moreover, the product of any $aN$, $bN \in A/N$ is a \(p\)-element by hypothesis. Hence, by \cref{lem:basic1}, $G/N = \gen{A/N}$ and, as $\abs{G/N}+ \abs{A/N} < \abs{G}+\abs{A}$, $G/N$ is soluble by induction.
\end{proof}

\begin{Lem}\label{lem:lem2}
	$A$ is a conjugacy class.
\end{Lem}

\begin{proof}
	If $A= B \cup C$ with $B$ and $C$ non-empty normal subsets of $G$.
	Then, as $\abs{B}+\abs{G} <\abs{A}+\abs{G}$ and $\abs{C}+ \abs{G}< \abs{A}+\abs{G}$, $\gen{B}$ and $\gen{C}$ are soluble normal subgroups of $G$ by the inductive hypothesis. Hence $\gen{B} \gen{C}=\gen{A}$ is soluble and normal in $G$, a contradiction.
\end{proof}

\begin{Lem}\label{lem:p odd}
	We have $p$ is odd.
\end{Lem}

\begin{proof}
  Suppose that $p=2$. Then for $x$, $y \in A$, $\gen{x,y}$ is a dihedral $2$-group. Hence, $\gen{A} \le F(G)$ by the Baer--Suzuki Theorem \cite[\nopp 39.6]{Aschbacher}. Since $G=\gen{A}$ by \cref{lem:basic1}, this is a contradiction. Hence $p$ is odd.
\end{proof}

\begin{Lem}	\label{lem:lem4}
	The group $G$ has a unique minimal normal subgroup $N$ and $C_G(N)=1$.
\end{Lem}

\begin{proof}
	Suppose that $N$ and $M$ are distinct minimal normal subgroups of $G$. Then $N$ and $M$ are proper subgroups of $G$ and neither $N$ nor $M$ is soluble by \cref{lem:lem3}. However $M \cong M/N\cap M \cong MN/N\le G/N$ which is soluble by \cref{lem:lem3}, a contradiction. Hence $G$ has a unique minimal normal subgroup. Let $N$ be the unique minimal normal subgroup of $G$. Since $C_G(N)$ is a normal subgroup of $G$, and $N$ is the unique minimal normal subgroup of $G$, if $C_G(N)\ne 1$, $N \le C_G(N)$ and this contradicts \cref{lem:lem3}.
\end{proof}

From now on, we fix the minimal normal subgroup $N$ of $G$. By \cref{lem:lem3} we know that $N$ is a direct product of isomorphic non-abelian simple groups.

\begin{Lem}\label{lem:lem5}
	Let $a \in A$. Then $G= \gen{a} N$ and $G/N$ is cyclic of order dividing $p$. Furthermore, $p$ divides $\abs{N}$.
\end{Lem}

\begin{proof}
	Let $H= \gen{a} N$ and assume that $H<G$. Then $ \gen{A \cap H}$ is soluble and normal in $H$ by \cref{lem:basic1}. But then, by \cref{lem:lem3},
	$$[\gen{A \cap H},N] \le \gen{A \cap H}\cap N = 1.$$
	But then $\gen{A \cap H} \le C_G(N)=1$ by \cref{lem:lem4}, a contradiction. Hence $G=H$ and $G/N$ is cyclic of order dividing $p$. Finally, \cref{thm:GN} implies that $p$ divides $\abs{N}$.
\end{proof}
%
%
%
%

Assume that $N$ is not simple. Then $G$ permutes the components of $N$ transitively. By \cref{lem:lem5} we have $\abs{G/N}=p$ and write $N= N_1\times \dots\times N_{p}$ where $N_1, \dots, N_{p}$ are non-abelian simple groups.

\begin{Lem}\label{lem:simp}
  $N$ is a non-abelian simple group.
\end{Lem}

\begin{proof} Suppose that $N$ is not simple. By \cref{prop:invert stuff} there exists a non-abelian subgroup  $H \le N_1$ of order coprime to $p$.  Then $$W=\gen{H,a} \cong H\wr \mathbb Z/p\mathbb Z.$$ Now \cref{lem:dih calc} gives that $(W\cap A)^2$ is not a set of \(p\)-elements, which is a contradiction.
\end{proof}

\begin{Lem}	\label{lem:3.5}
	Suppose \(H \le G\), $a \in A \cap N_G(H)$ and $G\ne H\gen{a}$. Then $a \in C_G(E(H))$.
\end{Lem}

\begin{proof}
	Suppose that $a \in N_G(H)$ and set $J=H\gen{a} \le G$. Then $J \ne G$ by hypothesis and so \cref{lem:basic1} implies $K=\gen{J \cap A}$ is soluble normal subgroup of $J$. But then $a \in J \le C_H(E(H))$ by \cite[\nopp 31.4]{Aschbacher}.
\end{proof}

\section{Finite groups of Lie type}	\label{sec:Lie type}

In this section $N$ will be a group of Lie type. Our first result exploits \cref{lem:GrossApp} and shows that $a\in A$ cannot induce field, graph-field or graph automorphisms on $N$. We follow the notational conventions of \cite{GLS3} for almost simple groups.

\begin{Lem}\label{lem:not field etc}
	Suppose that \(N\) is a simple group of Lie type. Then $G \le \mathrm{Inndiag}(N)$.
\end{Lem}

\begin{proof}
	Let $\GF(q)$ be the field of definition of $N$ where $q=r^c$ and $r$ is a prime. If $G \not \le \mathrm{Inndiag}(N)$, then, for $a\in A$, $aN$ contains a field, graph or graph-field automorphism. Let $S \in \Syl_p(G)$ be chosen so that $a \in S$. By \cref{lem:lem5}, $S\cap N \ne 1$. As $p$ is odd, \cref{lem:GrossApp} yields
	$$A \cap C_G(a) \not\subset \gen{a} $$
	and so $X =\gen{A \cap C_G(a)}$ is a normal soluble subgroup of $C_G(a)$ with $X> \gen{a}$. In particular, $X \cap N$ is a non-trivial soluble normal subgroup of $C_N(a)$.

	First consider the possibility that $aN$ contains either a field or a graph-field automorphism. Set $K_a= O^{r'}(C_N(a))$.
	Assume that $N \not \cong \SL_2(2^p)$, $\SL_2(3^p)$, $\SU_3(2^p)$, ${}^2\mathrm B_2(2^p)$. Then $K_a$ is a non-abelian simple group and $K_a=F^*(C_N(a))$ by \cite[Proposition 4.9.1 (a) and (b)]{GLS3}. Since $X$ is normal and soluble, we have $[X,K_a]\le X \cap K_a=1$. Hence $X\cap C_N(a)=1$, which is impossible as $\gen{a,b} \cap C_N(a)\ge \gen{ab^{-1}}$ where $b \in (A\cap C_G(a))\setminus \gen{a}$. So consider the ellipted possibilities for $N$. In each case we know that $p$ divides $\abs{C_K(a)}$ and thus we have $p =3$ in the first three cases and $p=5$ in the last case. For these groups we refer to \cref{prop:sporadics}. Hence $aN$ contains neither a field nor a graph-field automorphism.

	Assume that $aN$ has a graph automorphism. Then $p=3$ and $N \cong {}^3\mathrm D_4(q)$ or $\mathrm D_4(q)$.
	Assume additionally that $r=3$. Then \cite[Proposition 4.9.2(g)]{GLS3} tells us that either $a$ is a pure graph automorphism $\tau$, say, or $a= \tau z$ where $z$ is a long root element of $N$ contained in $C_N(\tau)$.
	Furthermore, \cite[Proposition 4.9.2 (b)(5)]{GLS3} yields
	$C_N(\tau) \cong \mathrm G_2(q)$. If $a=\tau$, then, as $X \cap N$ is a non-trivial soluble normal subgroup of $C_N(\tau)$, we have a contradiction. Thus $a= \tau z$. By \cite[Proposition 4.9.2 (b)(5)]{GLS3}, $C_N(\tau z)= C_{C_N(\tau)}(z) < C_{N}(\tau)$ and so $A\cap C_G(\tau) \not\subset \gen{\tau z}$. In particular, \(\gen{A \cap C_G(\tau)}\) is a non-trivial soluble normal subgroup of the simple group \(C_N(\tau)\). This contradiction shows that $r\ne 3$. The story is similar in this case, however we need to employ \cite[Table 4.7.3A]{GLS3}. We see that for both $N \cong {}^3\mathrm D_4(q)$ and $\mathrm D_4(q)$, one possibility is that $C_N(a) \cong \mathrm G_2(q)$ and otherwise
  $C_N(a)\cong \PGL_3(q)$ if $q \equiv 1 \pmod 3$ and $C_N(a)\cong \mathrm{PGU}_3(q)$ if $q\equiv 2 \pmod 3$. In all cases, we have $X \cap N$ is a non-trivial soluble normal subgroup of $C_N(a)$ and so we conclude that $q=3$ and $C_N(a) \cong \PGU_3(2)$ and $N \cong \mathrm D_4(2)\cong \Omega_8^+(2)$ or $\triality(2)$. Now we refer to \cref{prop:sporadics} to complete the proof.
\end{proof}

\begin{remark}
	While investigating the ``pure" triality automorphism of $N=\mathrm P\Omega_8^+(3) = \D_4(3)$, we happened upon a minor error in \cite[Theorem $\mathrm {A}^*$]{Guest}. Let $\tau$ be such an automorphism. Then a {\sc Magma} \cite{Magma} calculation reveals that $\gen{\tau, \tau ^x}$ is soluble for all $x\in N$. This example is not listed in \cite[Table 1]{Guest} whereas $\mathrm G_2(3)$ is listed.
	The fact that $\mathrm G_2(3)$ is present in \cite[Table 1]{Guest} invalidates the argument in the last lines of the proof of \cite[Lemma 7]{Guest}. The result for the general case is correct.
\end{remark}

\begin{Lem}\label{lem:notpara}
  Suppose that \(N\) is a simple group of Lie type and assume $G \le \mathrm{Inndiag}(N)$. If $P$ is a parabolic subgroup of $G$, then $A\cap P = \emptyset$.
\end{Lem}

\begin{proof}
  Let $B$ be a Borel subgroup of $G$ contained in $P$ and assume that $A\cap P \not=\emptyset$.
  By \cref{lem:basic1}, $ \gen{A \cap P}$ is a soluble normal subgroup of $P$. In particular, $B$ normalizes $ \gen{A \cap P}$ and so $\gen{A \cap P} B\ge B$ is a parabolic subgroup of $G$. From among all the parabolic subgroups containing $B$ choose $R$ so that $M = \gen{A \cap R} B$ has maximal order. Then $M$ is soluble. Suppose that $M$ is not a maximal parabolic subgroup of $G$. Then there exist maximal parabolic subgroups $L_1$ and $L_2$ with $M \le L_1\cap L_2$. Notice that $A \cap M \subset A \cap L_i$, $i=1,2$ and the maximal choice of $M$ implies that $A\cap L_1=A\cap M= A \cap L_2$.
  But then $\gen{A\cap L_1}=\gen{A \cap L_2}$ is normal in $\gen{L_1, L_2}=G$, which is a contradiction. Hence $M$ is a maximal parabolic subgroup. Now \cref{prop:sporadics} provides the contradiction.
\end{proof}

\subsection{Cross characteristic}

In this subsection we assume that $N$ is defined in characteristic $r$ with $r \ne p$. For the first lemma, we note that by \cref{lem:not field etc} we have $N \le G \le \mathrm{Inndiag}(N)$. Notice that the next lemma requires that $A \subset N$ so that Gow's Theorem may be applied.

\begin{Lem}\label{lem:GowApp}
	Suppose that $G=N$ is a simple Lie type group defined in characteristic $r \ne p$. Then $r$ divides $\abs{C_G(a)}$ for all $a\in A$.
\end{Lem}

\begin{proof}
	 Suppose not. As $\abs{C_G(a)}$ is coprime to $r$, according to Gow \cite{Gow}, $a$ is a regular semisimple element of $G$. Let $s$ be a prime dividing $\abs{G}$ with $r\ne s\ne p$. Then, by \cite[Theorem 2]{Gow}, there exists $b \in A$ such that $ab$ has order $s$, a contradiction. Hence $r$ divides $\abs{C_G(a)}$.
\end{proof}

\begin{Lem}\label{lem:crossRank1}
	The subgroup $N$ is not a rank $1$ Lie type group in characteristic $r \ne p$.
\end{Lem}

\begin{proof}
	By \cref{lem:not field etc}, $G$ can be identified as a subgroup of $\mathrm{Inndiag}(N)$. Suppose that the lemma is false.

	If $N \cong \PSL_2(r^c)$ or, $r=2$ and $N\cong {}^2\mathrm B_2(r^c)$, then $N=G$ and $C_G(x)$ is an $r$-group for every element $x\in G$ of order $r$. Thus, as $p\ne r$, $\abs{C_G(a)}$ is an $r'$-group for all $a \in A$, and this contradicts \cref{lem:GowApp}.

	Similarly, if $r=3$, $G= {}^2\mathrm G_2(3^c)$, $c \ge 3$, then, by \cite[Theorem (3)]{Ward} the centralizer of a $3$-element is either a $3$-group or a $\{2,3\}$-group and so this case is also eliminated by \cref{lem:GowApp} as $p$ is odd.

	Suppose that $N \cong \PSU_3(r^c)$. Then \cref{lem:GowApp} implies that $p$ divides $r^c+1$ and therefore $C_G(a) \cong \mathrm{GU}_2(r^c)$. In particular, $a$ is contained in a torus of order $(r^c+1)^2/\gcd(r^c+1,3)$ and so $A \cap C_G(a) \not \le Z(C_G(a))$ and from this we conclude that $r^c=r=3$ and we refer to \cref{prop:sporadics} to eliminate this case.
\end{proof}

\begin{Lem}\label{lem:notcross}
	We have $N$ is not a Lie type group in characteristic $r \ne p$.
\end{Lem}

\begin{proof}
Again we use \cref{lem:not field etc} to identify $G$ with a subgroup of $\mathrm{Inndiag}(N)$. We have that $G$ has rank at least $2$ by \cref{lem:crossRank1}.

Let $a\in A$ and $R \in \Syl_r(C_G(a))$. If $R\ne 1$, by the Borel--Tits Theorem \cite[Theorem 3.1.3]{GLS3}, $$a \in C_G(R) \le N_G(R) \le P$$ for some parabolic subgroup $P$ of $G$. This then contradicts \cref{lem:notpara}. Hence

\centerline{$\abs{C_G(a)}$ is not divisible by $r$. }\bigskip

\noindent Therefore
\cref{lem:GowApp} implies $A \subseteq G\setminus N$. Since $p$ is odd, \cref{lem:not field etc} and \cite[Theorem 2.5.12]{GLS3} imply that $N \cong \PSL_{dp}(r^b)$ or $\mathrm E_6(r^b)$ and $p$ divides $r^b-1$, or $N \cong \PSU_{dp}(r^b)$ or ${}^2\mathrm E_6(r^b)$ and $p$ divides $r^b+1$.

In all cases, a Sylow $p$-subgroup $S$ of $G$ is contained in $N_G(T)$ where $T$ is a torus and $N_G(T)/T \cong W$, the (untwisted) Weyl group of $G$. If $A \cap N_G(T) \not \subset T$, then $ \gen{A\cap N_G(T)} T/T$ is a soluble normal subgroup of $W$. We conclude that $W\cong \Sym(4)$ or $\Sym(3)$ and $p=3$, but then $\gen{A\cap N_G(T)} T/T\le W'$ whereas we have assumed $A \not\subseteq N$.
Thus $A \cap N_G(T) \subset T$.

If $N\cong \PSL_{dp}(r^b)$ or $\mathrm E_6(r^b)$, then $T$ is a maximally split torus and thus is conjugate to a subgroup of a Borel subgroup $B$. But then $A \cap B \not=\emptyset$, contrary to \cref{lem:notpara}. Hence $N \cong \PSU_{dp}(r^b)$ or ${}^2\mathrm E_6(r^b)$.

Assume that $N \cong \PSU_{dp}(r^b)$. Let $L$ be the image of the stabiliser of a non-degenerate $1$-space which is stabilised by $T$. Then, as $dp\ne 4$ by \cref{lem:crossRank1}, $F^*(L)\cong \SU_{dp-1}(r^b)$ is quasisimple.
Since $T$ normalizes $A$, we have $A\cap N_G(L) \ne \emptyset$.
Hence \cref{lem:3.5} shows that $L$ is centralized by $\gen{A\cap N_G(L)} > 1$. In particular, each $a \in A$ centralizes an $r$-element, a contradiction.

If $N \cong {}^2\E_6(q)$, then, using \cite[Table 5.1]{LiebeckSaxlSeitz}, $T$ normalizes a subgroup $L$ isomorphic to $2\cdot(\PSL_2(r^b)\times\PSU_6(r^b))$. Hence \cref{lem:3.5} implies $\gen{A\cap L}$ centralizes an $r$-element.
This contradiction completes the proof.
\end{proof}

\subsection{Defining characteristic}

In this subsection we consider the possibility that $N$ is a Lie type group defined in characteristic $p$. In particular, as $p$ is odd, $N$ is defined in odd characteristic.

\begin{Lem}\label{lem:notPSL2p}
	We have $N$ is not the simple group $\PSL_2(p^c)$.
\end{Lem}

\begin{proof}
	Suppose that $N \cong \PSL_2(p^c)$. We calculate in $\widehat N \cong \SL_2(p^c)$. Let $\widehat A$ be a preimage of $A$ in $\GL_2(p^c)$ consisting of elements of order $p$. Pick $\widehat a\in \widehat A$.
  Since all elements of order $p$ in $\GL_2(p^b)$ are conjugate, we may take $\widehat a= \left(\begin{smallmatrix}1&1\\0&1\end{smallmatrix}\right)$. For $\lambda\in \GF(p^c)$ be non-zero set
	$\widehat c= \left(\begin{smallmatrix}0&\lambda\\-\lambda^{-1}&0\end{smallmatrix}\right)$.
	Then $\widehat a\widehat a^{\widehat c}= \left(\begin{smallmatrix} 1-\lambda^{-2}&1\\-\lambda^{-2}&1\end{smallmatrix}\right)$ and
	$$
		(\widehat a\widehat a^{\widehat c})^2=
		\begin{pmatrix}
			1 - 3 \lambda^{-2} + \lambda^{-4}	&	2 - \lambda^{-2}	\\
			\lambda^{-4} - 2 \lambda^{-2}		&	1 - \lambda^{-2}	\\
		\end{pmatrix}.
	$$
  Since $N$ is a simple group, $p^c> 3$ and so $\lambda$ can be chosen so that $(\widehat a\widehat a^{\widehat c})^2$ does not have trace $2$. As elements of order $p$ have trace $2$, we can choose $\lambda$ so that $\widehat a\widehat a^{\widehat c}$ does not have order $p$ or $2p$, this proves the result.
\end{proof}

\begin{Lem}\label{lem:notU3p}
  We have	$N \not\cong\PSU_3(p^c)$.
\end{Lem}

\begin{proof}
	Suppose that $N \cong \PSU_3(p^c)$ and consider $N$ as a subgroup of $K \cong \mathrm{PGU}_3(p^c)$. Then $K$ has exactly two conjugacy classes of elements of order $p$, one of which is conjugate into a subgroup of $\mathrm{SU}_2(p^c)\cong \SL_2(p^c)$ and the other into $\mathrm{SO}_3(p^c) \cong \PSL_2(p^c)$. Hence \cref{lem:notPSL2p} implies that $p^c=p=3$, and this case is handled in \cref{prop:sporadics} as $\PSU_3(3)\cong \G_2(2)'$.
 \end{proof}

\begin{Lem} \label{lem:notReesmall} We have
	$N \not\cong{}^2\mathrm G_2(3^c)$ with $c \ge 3$.
\end{Lem}

\begin{proof}
	If $N \cong {}^2\mathrm G_2(3^c)$ with $c\ge 3$ and odd, then $N$ has a subgroup $H\cong {}^2\mathrm G_2(3)\cong \SL_2(8){:}3$ and every conjugacy class of 3-elements in $N$ intersects $H$ non-trivially by \cite[Table 22.2.7]{LiebeckSeitz}. Hence $\gen{A \cap H} $ is a non-trivial soluble normal subgroup of $H$, which is a contradiction.
\end{proof}

\begin{Lem}\label{lem:char p generic}
  The group	$N$ is not a simple Lie type group defined in characteristic $p$ of rank at least $2$.
\end{Lem}

\begin{proof}
  By \cref{lem:not field etc}, we can identify $G$ with a subgroup of $\mathrm{Inndiag}(G)$. Let $B$ be a Borel subgroup of $G$. Since $B$ contains a Sylow $p$-subgroup of $G$, $A\cap B\not=\emptyset$. The result now follows from \cref{lem:notpara}.
\end{proof}

\section{Alternating groups, sporadic simple groups and computer calculations}	\label{sec:computer}

To complete our analysis of a minimum counterexample to \cref{thm:main} we need to investigate the case when $N$ is an alternating group or a sporadic simple group. We also need to eliminate computationally the small number of groups that have been excluded in the analysis in \cref{sec:Lie type}.

\begin{Lem}\label{lem:notalt}
  We have	$N$ is not an alternating group.
\end{Lem}

\begin{proof}
	Assume that $N$ is an alternating group $\Alt(n)$. For $n \le 6$, we check the result by elementary calculation (see also \cref{prop:sporadics}). Since $p$ is odd by \cref{lem:p odd}, $a\in A\subset N=G$. Suppose that $n > p$ and $p \ge 5$, then we may assume that $a$ normalizes and does not centralize the subgroup $\Alt(\{1, \dots, p\})$, but this contradicts \cref{lem:3.5}. Therefore, $p=3$ and, as $n> 6$, we may assume that $a$ normalizes, but does not centralize $Y=\Alt(\{1,\dots,6\})$ with $a=yz$, where $y \in Y$ and $z \in \Alt(\{7,\dots,n\})$. But then applying the result with $n=6$ gives a contradiction. Hence we have that $n=p>5$. Assume that $a=(1,2,3, \dots, p)$ and let $b=(3,1,4)$. Then $a^b= (4,2,1,3,\dots, p)$ and $aa^b$ fixes $2$. Thus, $aa^b$ is a non-trivial $p$-element of order dividing $(p-1)!$, a contradiction. Hence $N$ is not an alternating group.
\end{proof}

For the computations in the next result we have used \cite{Magma} and \cite{GAP4}.

\begin{Prop}\label{prop:sporadics}
	The subgroup $N$ cannot be any of the following groups:
	\begin{enumerate}
    \item[(i)] One of the groups listed in \cref{lem:solvpara}.
    \item[(ii)] $\PSL_2(2^2)$, $\PSL_2(2^3)$, $\PSL_2(3^3)$, $\PSU_3(2^3)$ with $p=3$ and $\Sz(2^5)$ with $p=5$.
    \item[(iii)] A sporadic simple group.
	\end{enumerate}
\end{Prop}

\begin{proof}
	For this, we use the {\sc GAP} function included in \cref{sec:appendix}. This function takes as input a character table, group or string representing the {\sc ATLAS} name of a group. The character tables for the groups listed in the proposition are stored in {\sc GAP} 4.12.0 except for \(\PSU_5(3).2\) and \(\triality(3).3\) whose character tables can be computed in {\sc Magma} (the latter group is \texttt{AutomorphismGroupSimpleGroup(<13,4,3>)}) and imported into GAP. The function checks whether the square of a conjugacy class of elements of odd prime order \(p\) contains any elements which are not \(p\)-elements and applying it to all of the groups mentioned in the statement we see that this always happens, completing the proof.
\end{proof}
\section{The proofs of the main theorems and corollaries}\label{sec:proofs}

We begin this section by completing the proof of \cref{thm:main}.

\begin{proof}[Proof of \cref{thm:main}]
  With the notation developed in \cref{sec:reduction}, \cref{lem:simp} implies $N$ is a non-abelian simple group. By \cref{lem:notcross,lem:char p generic}, $N$ is not a Lie type group while \cref{lem:notalt} and \cref{prop:sporadics} show that $N$ is not an alternating group or a sporadic simple group. Hence there are no counterexamples to \cref{thm:main} and the proof is complete.
\end{proof}

  We can now prove our main theorem.

 \begin{proof}[Proof of \cref{thm:A}]
  By \cref{thm:main}, $Q=\gen{A}$ is a soluble group and by hypothesis $O_p(G)=1$. Let $a \in A$ and let $X$ be the smallest subnormal subgroup of $Q$ which contains $a$. Then $X= \gen{a^{X}}$. Since \(X\) is subnormal in \(G\), \(O_p(X) \leq O_p(G) = 1\) and so \cref{lem:class generated} tells us that \(X\) is a Frobenius group with complement \(\gen{a}\).
  Because Frobenius kernels are nilpotent, $X= F(X)\gen{a}$ and so, as $X$ is subnormal in $Q$, $F(X)\le F(Q)$. Set $\ov G= G/F(Q)$. Then $\ov{X}= \ov{\gen{a}}$ is subnormal in $\ov{Q}$. Thus $\ov{a} \in O_p(\ov{Q})$. Since this is true for all $a\in A$, we deduce that $\ov{Q}={\gen{\ov A}}\le O_p(\ov Q) \le \ov Q$. Therefore, $\ov Q$ is a $p$-group.
  We now claim that $\ov Q= \gen{\ov A}$ is elementary abelian.
  If $\ov a\in \ov A$ normalizes but does not centralize some elementary abelian $p$-subgroup of $\ov Q$, then \cref{lem:is abundant} implies that there exists $\ov Y \le \ov Q$ with $\abs{\ov a^{\ov Q}\cap \ov Y}=p$ and $\ov Y \cap \ov Q'\ne 1$ which contradicts
  \cref{lem:not abundant}. Hence every $\ov a\in \ov A$ centralizes every elementary abelian subgroup which it normalizes. Now \cref{lem:generation abelian} implies $\ov Q$ is elementary abelian. This completes the proof.
 \end{proof}

 \begin{proof}[Proof of \cref{cor:cor11}]
  By \cref{thm:A}, $Q=\langle A\rangle$ has elementary abelian Sylow $p$-subgroups Since $A^2 \subset Q$, every non-trivial element of $A^2$ has order $p$. So suppose that $a\in A$ and $a^{-1} \in A$. Since $F(Q)$ is a $p'$ group, there exists $x \in F(Q)$ such that $a^x \ne a$ as $C_Q(F(Q))\le F(Q)$. Hence $1 \ne [a,x]=a^{-1}a^x\in A^2$. However $[a,x]\in F(Q)$ is a $p'$-element. Therefore every element of $A^2$ has order $p$.
 \end{proof}

  \begin{proof}[Proof of \cref{cor:cor3}]
    By \cref{thm:A}, $G=\langle A\rangle$ is soluble and $p$ is odd. Therefore \cref{lem:class generated} implies that $G$ is a Frobenius group with complement $\gen{a}$. Finally, \cref{cor:cor11} completes the proof.
 \end{proof}

\appendix
\section{\sc{GAP} code} \label{sec:appendix}
This appendix contains the {\sc GAP} function used in \cref{prop:sporadics}.
\begin{verbatim}
ClassSquareTest := function(G)
local ans, CC, CT, Cx, i, n, ord, p, tmp;
if IsCharacterTable(G) then CT := G;
  else CT := CharacterTable(G);
fi;
n := NrConjugacyClasses(CT);
ord := OrdersClassRepresentatives(CT);
CC := Filtered([1..n],x->IsPrime(ord[x]) and ord[x]<>2);
for i in CC do
  p := PrimeDivisors(ord[i])[1];
  Cx := Filtered([1..n], x ->
    PrimeDivisors(ord[x]) <> [p]);
  tmp := Filtered(Cx, y ->
    ClassMultiplicationCoefficient(CT, i, i, y) = 0);
  if tmp = Cx then
    Add(ans, tmp);
  fi;
od;
if ans = [] then return true;
else Print(ans,"\n"); return false;
fi;
end;
\end{verbatim}
	%

\printbibliography

\end{document}